\documentclass[12pt]{amsart}
\usepackage{amscd,amsmath,amsthm,amssymb}
\usepackage{tikz}
\tikzstyle{punkt}=[circle, fill=black, minimum size=1mm,inner sep=0pt, draw]

\usepackage{amsfonts,amssymb,amscd,amsmath,enumerate,verbatim}

\usetikzlibrary{arrows}
\input xy
\xyoption{all}
\def\frk{\mathfrak}               

\def\Phi{{\frk N}}
\def\opn#1#2{\def#1{\operatorname{#2}}} 
\opn\chara{char} \opn\length{\ell} \opn\pd{pd} \opn\rk{rk}
\opn\projdim{proj\,dim} \opn\injdim{inj\,dim} \opn\rank{rank}
\opn\depth{depth} \opn\grade{grade} \opn\height{height}
\opn\embdim{emb\,dim} \opn\codim{codim}

\opn\Tr{Tr} \opn\bigrank{big\,rank}
\opn\superheight{superheight}\opn\lcm{lcm}
\opn\trdeg{tr\,deg}
\opn\reg{reg} \opn\lreg{lreg} \opn\ini{in} \opn\lpd{lpd}
\opn\size{size}\opn{\mult}{mult}
\opn\div{div} \opn\Div{Div} \opn\cl{cl} \opn\Cl{Cl}
\opn\Spec{Spec} \opn\Supp{Supp} \opn\supp{supp} \opn\Sing{Sing}
\opn\Ass{Ass} \opn\Min{Min}
\opn\Ann{Ann} \opn\Rad{Rad} \opn\Soc{Soc}
\opn\Syz{Syz} \opn\Im{Im} \opn\Ker{Ker} \opn\Coker{Coker}
\opn\Am{Am} \opn\Hom{Hom} \opn\Tor{Tor} \opn\Ext{Ext}
\opn\End{End} \opn\Aut{Aut} \opn\id{id} \opn\ini{in}

\opn\nat{nat}
\opn\pff{pf}
\opn\Pf{Pf} \opn\GL{GL} \opn\SL{SL} \opn\mod{mod} \opn\ord{ord}
\opn\Gin{Gin}
\opn\Hilb{Hilb}\opn\adeg{adeg}\opn\std{std}\opn\ip{infpt}
\opn\Pol{Pol}
\opn\sat{sat}
\opn\Var{Var}
\opn\Gen{Gen}

\opn\aff{aff} \opn\con{conv} \opn\relint{relint} \opn\st{st}
\opn\lk{lk} \opn\cn{cn} \opn\core{core} \opn\vol{vol}
\opn\link{link} \opn\star{star}
\opn\gr{gr}

\def\pot#1#2{#1[\kern-0.28ex[#2]\kern-0.28ex]}

\opn\dirlim{\underrightarrow{\lim}}
\opn\inivlim{\underleftarrow{\lim}}
\def\Implies{\ifmmode\Longrightarrow \else
        \unskip${}\Longrightarrow{}$\ignorespaces\fi}
\def\implies{\ifmmode\Rightarrow \else
        \unskip${}\Rightarrow{}$\ignorespaces\fi}
\def\iff{\ifmmode\Longleftrightarrow \else
        \unskip${}\Longleftrightarrow{}$\ignorespaces\fi}

\let\:=\colon
\newtheorem{Theorem}{Theorem}[section]
\newtheorem{Lemma}[Theorem]{Lemma}
\newtheorem{Corollary}[Theorem]{Corollary}

\newtheorem{Remark}[Theorem]{Remark}

\newtheorem{Example}[Theorem]{Example}

\let\epsilon\varepsilon
\let\phi=\varphi
\let\kappa=\varkappa
\def\qed{\ifhmode\textqed\fi
      \ifmmode\ifinner\quad\qedsymbol\else\dispqed\fi\fi}
\def\textqed{\unskip\nobreak\penalty50
       \hskip2em\hbox{}\nobreak\hfil\qedsymbol
       \parfillskip=0pt \finalhyphendemerits=0}
\def\dispqed{\rlap{\qquad\qedsymbol}}

\opn\dist{dist}
\def\pnt{{\raise0.5mm\hbox{\large\bf.}}}

\opn\Lex{Lex}
\opn\diam{diam}

\begin{document}
\title{Diameter and connectivity of finite simple graphs}
\author[T.~Hibi]{Takayuki Hibi}
\address[Takayuki Hibi]
{Department of Pure and Applied Mathematics, 
Graduate School of Information Science and Technology, 
Osaka University, 
Suita, Osaka 565-0871, Japan}
\email{hibi@math.sci.osaka-u.ac.jp}
\author[S.~S.~Madani]{Sara Saeedi Madani}
\address[Sara Saeedi Madani]
{Department of Mathematics and Computer Science, Amirkabir University of Technology, Tehran, Iran, and School of Mathematics, Institute for Research in Fundamental Sciences, Tehran, Iran} 
\email{sarasaeedi@aut.ac.ir}
\subjclass[2010]{Primary 05E40; Secondary 05C75}
\keywords{Finite simple graphs, binomial edge ideals, vertex connectivity, diameter, free vertices.}
\thanks{
The first author was partially supported by JSPS KAKENHI 19H00637. The research of the second author was in part supported by a grant from IPM (No. 99130113).
}
\begin{abstract}
Let $G$ be a finite simple non-complete connected graph on $\{1, \ldots, n\}$ and $\kappa(G) \geq 1$ its vertex connectivity.  Let $f(G)$ denote the number of free vertices of $G$ and $\diam(G)$ the diameter of $G$.  Being motivated by the computation of the depth of the binomial edge ideal of $G$, the possible sequences $(n, q, f, d)$ of integers for which there is a finite simple non-complete connected graph $G$ on $\{1, \ldots, n\}$ with $q = \kappa(G), f = f(G), d = \diam(G)$ satisfying $f + d = n + 2 - q$ will be determined. Furthermore, finite simple non-complete connected graphs $G$ on $\{1, \ldots, n\}$ satisfying $f(G) + \diam(G) = n + 2 - \kappa(G)$ will be classified. 
\end{abstract}
\maketitle
\section*{Introduction}

Let $G$ be a finite simple connected graph on the vertex set $[n] = \{1, \ldots, n\}$ and the edge set $E(G)$.  The {\em distance} $\dist_G(i,j)$ of $i$ and $j$ in $[n]$ with $i \neq j$ is the smallest length of paths connecting $i$ and $j$ in $G$.  Especially, if $\{i, j\} \in E(G)$, then $\dist_G(i,j) = 1$.  The {\em diameter} of $G$ is 
\[
\diam(G) = \max \{\dist_G(i,j) : i, j \in [n] \}.
\]
The {\em induced subgraph} of $G$ on $T \subset [n]$ is the subgraph $G_T$ of $G$ on the vertex set $T$ whose edges are those $\{i, j\} \in E(G)$ with $ i \in T$ and $j \in T$.  Let $\kappa(G) \geq 1$ denote the {\em vertex connectivity} of $G$.  In other words, $\kappa(G)$ is the smallest cardinality of $T \subset [n]$ for which 
$G_{[n]\setminus T}$ is disconnected.  A vertex $i \in [n]$ is {\em free} if $\{i, j\} \in E(G)$ and $\{i, j'\} \in E(G)$ with $j \neq j'$, yield $\{j, j'\} \in E(G)$. In particular, a free vertex belongs to exactly one maximal clique (i.e., complete subgraph) of $G$. Let $f(G)$ denote the number of free vertices of $G$. 

Let $S = K[x_1, \ldots, x_n, y_1, \ldots, y_n]$ denote the polynomial ring in $2n$ variables over a field $K$ and let 
\[
J_G = (x_iy_j - x_jy_i : \{i, j\} \in E(G))
\]
be the {\em binomial edge ideal} of $G$, which was introduced in \cite{HHHKR} and \cite{Ohtani} independently.  It follows from \cite{BN} and \cite{RSK} that, when $G$ is non-complete, one has
\[
f(G) + \diam(G) \leq \depth(S/J_G) \leq n + 2 - \kappa(G),
\]
where $\depth(S/J_G)$ is the depth of $S/J_G$.  It is then a reasonable problem to find those finite simple non-complete connected graphs $G$ on $[n]$ for which  
\[
f(G) + \diam(G) = n + 2 - \kappa(G).
\]

In the present paper, the possible sequences $(n, q, f, d)$ of integers for which there is a finite simple non-complete connected graph $G$ on $[n]$ with $q = \kappa(G)$, $f = f(G)$ and $d = \diam(G)$ satisfying $f + d = n + 2 - q$ will be determined. Furthermore, finite simple non-complete connected graphs $G$ on $[n]$ satisfying $f(G) + \diam(G) = n + 2 - \kappa(G)$ will be classified.

Throughout this paper, all graphs are finite and simple, and we denote the vertex set and the edge set of a finite simple graph $G$ by $V(G)$ and $E(G)$, respectively, unless we mention something else. 

\section{Sequences of integers}
The possible sequences $(n, q, f, d)$ of integers for which there is a finite simple non-complete connected graph $G$ on $[n] = \{1, \ldots, n\}$ with $q = \kappa(G)$, $f = f(G)$ and $d = \diam(G)$ satisfying $f + d = n + 2 - q$ will be determined. 
 
\begin{Lemma}
\label{tehran}
Let $G$ be a finite simple connected graph on $[n]$ with $\mathrm{diam}(G)\geq 3$ and $\kappa(G)\geq 2$. Then
\[
f(G) + \diam(G) < n + 2 - \kappa(G).
\] 
\end{Lemma}

\begin{proof}
Suppose on contrary that $f(G) + \diam(G) = n + 2 - \kappa(G)$.
Let $d = \diam(G) \geq 3$ and let
\[
v, u_1, u_2, \ldots, u_{d-1}, w
\]
be a path between $v$ and $w$ for which each of $u_i$ is non-free.  

Now, the number of non-free vertices of $G$ is
\[
n - f(G) = \diam(G) + \kappa(G) - 2 = (d - 1) + (\kappa(G) - 1).
\]
Let $x_1, \ldots, x_q$ be the non-free vertices of $G$ belonging to $$[n] \setminus \{v, u_1, u_2, \ldots, u_{d-1}, w\}.$$  Then $q \leq \kappa(G) - 1$.  

Let $G'$ denote the subgraph obtained from $G$ by removing $x_1, \ldots, x_{q-1}$. If $q\geq 2$, then it follows that 
\[
\kappa(G') \geq \kappa(G) - (q - 1) \geq \kappa(G) - (\kappa(G) - 2) = 2.
\]
In particular, when $q \in \{0, 1\}$, one has $G' = G$. 
Since each of the free vertices of $G$ remains free in $G'$, it follows that each non-free vertex of $G'$ belongs to $\{x_q, v, u_1, u_2, \ldots, u_{d-1}, w\}$.  Since $d \geq 3$, we have $\{x_q, v\} \not\in E(G)$ or $\{x_q, w\} \not\in E(G)$. We may assume that $\{x_q, v\} \not\in E(G)$.  Since $\kappa(G') \geq 2$, there is a neighbor $v_0$ of $v$ with $v_0 \neq u_1$ and every neighbor $v_0$ of $v$ with $v_0 \neq u_1$ is free.  Since $\kappa(G') \geq 2$, the subgraph $G' - u_1$ is connected.  Let
\[
v, v_0, v_1, \ldots, v_{d'-2}, w
\]
be a shortest path between $v$ and $w$ in $G' - u_1$, where $d' \geq d \geq 3$.  However, since $v_0$ is a neighbor of $v$ and is free, one has $\{v, v_1\} \in E(G)$.  This contradicts the fact that the path $v, v_0, v_1, \ldots, v_{d'-2}, w$ is shortest. 
\end{proof}


\begin{Example}
\label{q=1}
{\rm
Let $d \geq 2$ and $f \geq 2$ be integers.  Let $P_{d+1}$ denote the path on $\{1, \ldots, d+1\}$ and $\Gamma$ the finite simple connected graph on $[n]$ with $n = d + 1 + f - 2$ obtained from $P_{d+1}$ by adding $f - 2$ leaves
\[
\{2, d + 2\}, \{2, d + 3\}, \ldots, \{2, d + 1 + f - 2\}
\]
to $P_{d+1}$.  It then follows that
\[
\kappa(\Gamma) = 1, \, \, \, f(\Gamma) = f, \, \, \, \diam(\Gamma) = d.
\]
Thus, in particular,
\[
f(\Gamma) + \diam(\Gamma) = n + 2 - \kappa(\Gamma).
\]
The graph $\Gamma$ is depicted in Figure~\ref{Gamma}.
}
\end{Example}

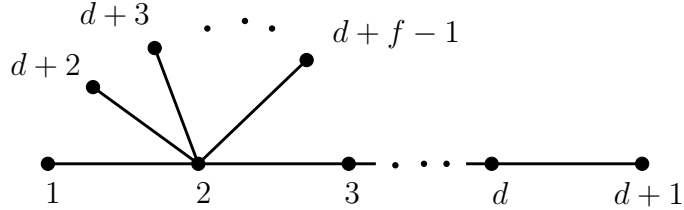
\begin{figure}[h!]
	\centering
	\begin{tikzpicture}[scale=1]
\draw [line width=1.1pt] (0,1)-- (2,1);
\draw [line width=1.1pt] (2,1)-- (4,1);
\draw [line width=1.1pt] (7.9,1)-- (5.9,1);
\draw [line width=1.1pt] (2,1)-- (1.42,2.54);
\draw [line width=1.1pt] (0.6,2.02)-- (2,1);
\draw [line width=1.1pt] (2,1)-- (3.44,2.38);
\draw [line width=1.1pt] (4,1)-- (4.36,1);
\draw [line width=1.1pt] (5.9,1)-- (5.56,1);
\draw (-0.18,0.9) node[anchor=north west] {$1$};
\draw (5.76,0.9) node[anchor=north west] {$d$};
\draw (7.38,0.9) node[anchor=north west] {$d+1$};
\draw (3.8,0.9) node[anchor=north west] {$3$};
\draw (1.82,0.9) node[anchor=north west] {$2$};
\draw (-0.64,2.62) node[anchor=north west] {$d+2$};
\draw (0.28,3.3) node[anchor=north west] {$d+3$};
\draw (3.64,3.06) node[anchor=north west] {$d+f-1$};
\begin{scriptsize}
	\draw [fill=black] (0,1) circle (2.5pt);
	\draw [fill=black] (2,1) circle (2.5pt);
	\draw [fill=black] (4,1) circle (2.5pt);
	\draw [fill=black] (5.9,1) circle (2.5pt);
	\draw [fill=black] (7.9,1) circle (2.5pt);
	\draw [fill=black] (4.62,0.98) circle (1pt);
	\draw [fill=black] (5,1) circle (1pt);
	\draw [fill=black] (5.3,1) circle (1pt);
	\draw [fill=black] (0.6,2.02) circle (2.5pt);
	\draw [fill=black] (1.42,2.54) circle (2.5pt);
	\draw [fill=black] (3.44,2.38) circle (2.5pt);
	\draw [fill=black] (2.12,2.8) circle (1pt);
	\draw [fill=black] (2.62,2.9) circle (1pt);
	\draw [fill=black] (2.98,2.8) circle (1pt);
\end{scriptsize}
\end{tikzpicture}
\caption{The graph $\Gamma$}
\label{Gamma}
\end{figure}

\medskip
Let $G$ and $H$ be two graphs on vertex sets $V(G)$ and $V(H)$, respectively, with $V(G)\cap V(H)=W$. Then the \emph{union} of $G$ and $H$, denoted by $G\cup H$, is the graph on the vertex set $V(G)\cup V(H)$ and the edge set $E(G)\cup E(H)$.

\begin{Example}
\label{d=2}
{\rm
Let $q \geq 2$, $s \geq 1$ and $t \geq 1$ be integers. Let 
$K_{q+s}$ denote the complete graph on $$\{1, \ldots, q, q+1, \ldots, q+s\},$$ 
and 
$K'_{q+t}$ the complete graph on $$\{1, \ldots, q, q+s+1, \ldots, q+s+t\}.$$
Let $\Omega = K_{q+s} \cup K'_{q+t}$ on $[n]$ with $n = q+s+t$.  It then follows that
\[
\kappa(\Omega) = q, \, \, \, f(\Omega) = s + t, \, \, \, \diam(\Omega) = 2.
\]
Thus in particular
\[
f(\Omega) + \diam(\Omega) = n + 2 - \kappa(\Omega).
\]
Denoting the complete graph on 
\[
\{1,\ldots,q\},
\]
by $K''_q$, a sketch of the graph $\Omega$ is depicted in Figure~\ref{omega}.
}
\end{Example}

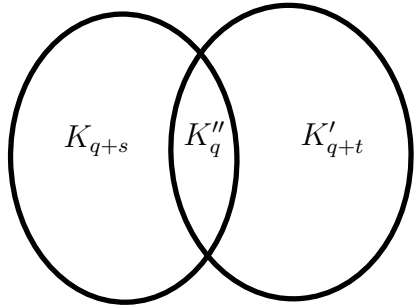
\begin{figure}[h!]
	\centering
	\begin{tikzpicture}[scale=0.9]
\draw (-8.1,1.03) node[anchor=north west] {$K_{q+s}$};
\draw (-4.6,1.09) node[anchor=north west] {$K'_{q+t}$};
\draw (-6.33,1.09) node[anchor=north west] {$K''_q$};
\draw [rotate around={-88.6980473274229:(-7.05,0.39)},line width=2pt] (-7.05,0.39) ellipse (2.13023970952485cm and 1.6717120625384387cm);
\draw [rotate around={-87.25191181994722:(-4.58,0.48)},line width=2pt] (-4.58,0.48) ellipse (2.1701435390987793cm and 1.7729701013531438cm);
\end{tikzpicture}
\caption{The graph $\Omega$}
\label{omega}
\end{figure}

\begin{Theorem}
\label{Boston}
Given integers $n \geq 3$, $q \geq 1$, $f \geq 0$ and $d \geq 1$, satisfying $$f + d = n + 2 - q,$$ there exists a finite simple non-complete connected graph $G$ on $[n] = \{1, \ldots, n\}$ with $$q = \kappa(G), \, \, \, f = f(G), \, \, \, d = \diam(G)$$ if and only if one has either
$$n \geq 3, \, \, \, q = 1, \, \, \, f \geq 2, \, \, \, d \geq 2$$
or
$$n \geq 4, \, \, \, q \geq 2, \, \, \, f \geq 2, \, \, \, d = 2.$$
\end{Theorem}

\begin{proof}
By virtue of Lemma \ref{tehran}, what we must show is the existence of a finite simple non-complete connected graph $\Gamma$ on $[n]$ with
\[
n \geq 3, \, \, \, \kappa(\Gamma) = 1, \, \, \, f(\Gamma) \geq 2, \, \, \, \diam(\Gamma) \geq 2, \, \, \, f(\Gamma) + \diam(\Gamma) = n + 1
\]
as well as of a finite simple non-complete connected graph $\Omega$ on $[n]$ with
\[
n \geq 4, \, \, \, \kappa(\Omega) \geq 2, \, \, \, f(\Omega) \geq 2, \, \, \, \diam(\Omega) = 2, \, \, \, f(\Omega) = n - \kappa(\Omega).
\]
Now, the existence of $\Gamma$ and $\Omega$ is guaranteed by Examples \ref{q=1} and \ref{d=2}. \, \, \, \, \, \, \, \, \, 
\end{proof}

\section{Classification}
Now, Theorem \ref{Boston} makes us turn to the problem of classifying finite simple non-complete connected graphs $G$ on $[n] = \{1, \ldots, n\}$ satisfying $$f(G) + \diam(G) = n + 2 - \kappa(G).$$

Recall that an induced path (resp. cycle) in a graph $G$ is a path (resp. cycle) which is an induced subgraph of $G$. A \emph{chordal} graph is a graph which has no induced cycle of length bigger than~3. By a well-known result due to Dirac~\cite{Di}, (see also \cite[Theorem~9.2.12]{HH}), a graph is chordal if and only if it admits a \emph{perfect elimination ordering}, that is its vertices can be labeled as $1,\ldots,n$ such that for all $j\in [n]$, the set $C_j=\{i: i\leq j\}$ induces a clique in $G$, (see also \cite[p.~172]{HH}).  

\medskip
In the following two steps, we try to find out more about what the two possible situations mentioned in Theorem~\ref{Boston} imply. 

\medskip
   
\noindent
{\bf (First Step)}\, $n \geq 3, \, \, \, \kappa(G) = 1, \, \, \, f(G) \geq 2, \, \, \, \diam(G) \geq 2$.

Let $d = \diam(G) \geq 2$ and let 
\[
v, u_1, u_2, \ldots, u_{d-1}, w
\]
be a path between $v$ and $w$ for which each of $u_i$ is non-free.  Since  
\[
f(G) = n + 1 - \diam(G) = n - (\diam(G) - 1),
\]
each vertex of $G$ belonging to $[n] \setminus \{u_1, u_2, \ldots, u_{d-1}\}$ is free.  If $v_0 \in [n]$ is free and $v_0$ is a neighbor of neither $v$ nor $w$, then the shortest path between $v$ and $v_0$ must be of the form
\[
v, u_1, u_2, \ldots, u_{i-1}, v_0.
\]
Thus in particular $v_0$ is a neighbor of $u_{i-1}$. The only other vertex in $\{u_1, u_2, \ldots, u_{d-1}\}$ which can be a neighbor of $v_0$ is $u_{i}$. 
If $v_0$ is a neighbor of $v$, then $v_0$ is a neighbor of $u_1$, and $u_1$ is the only neighbor of $v_0$ in $\{u_1,\ldots,u_{d-1}\}$. 
As a result, each vertex of $G$ belonging to $[n] \setminus \{u_1, u_2, \ldots, u_{d-1}\}$ is free and is a neighbor of at most two (consecutive) $u_{i}$'s. 
It then turns out that $G$ has a perfect elimination ordering, 
and hence $G$ is a chordal graph \cite[Theorem~9.2.12]{HH}.      

\medskip

\noindent
{\bf (Second Step)}\, $n \geq 4, \, \, \, \kappa(G) \geq 2, \, \, \, f(G) \geq 2, \, \, \, \diam(G) = 2$.

Let $q = \kappa(G)$.  Let $v$ and $w$ be vertices of $G$ with $\{v, w\} \not\in E(G)$.  It follows (\cite[p.~76]{B}) that there exist $q$ independent (i.e. vertex-disjoint) induced paths $P_1, \ldots, P_q$ between $v$ and $w$.  Let
\[
P_i : v, u^{(i)}_1, \ldots, u^{(i)}_{j_i - 1}, w.
\]
Each $j_i \geq 2$ and each $u^{(i)}_j$ is non-free.  Since the number of free vertices of $G$ is $f(G) = n - \kappa(G) = n - q$, it follows that the number of non-free vertices is equal to $q$.  Hence each $j_i = 2$.  Let $u_1, \ldots, u_q$ denote the non-free vertices of $G$ and
\[
P_i : v, u_i, w.
\]
Each vertex belonging to $[n] \setminus \{u_1, \ldots, u_q\}$ is free and a neighbor of each $u_i$.  By \cite[Theorem~9.2.12]{HH}, it then turns out that $G$ has a perfect elimination ordering
and that $G$ is a chordal graph.  

\medskip
From the two above steps, it follows that:
\begin{Theorem}
\label{chordal}
Every finite simple non-complete connected graph $G$ on $[n]$  satisfying $f(G) + \diam(G) = n + 2 - \kappa(G)$ is chordal.
\end{Theorem}

Next, we determine which chordal graphs satisfy the desired properties, by giving the precise classification. For this purpose, first we construct two different types of families of graphs which play an important role in the classification. 

\medskip
Let us fix a notation. If $\mathcal{T}$ is a family of finite simple graphs, then we set
\[
V(\mathcal{T}):=\cup_{G\in \mathcal{T}} V(G).
\] 

Also, recall that the set of neighbors (i.e. adjacent vertices) of a vertex $v$ in a graph $G$ is denoted by $N_G(v)$. 

\medskip
\noindent
{\bf The families $\mathcal{G}_d$:} Let $d\geq 2$ be an integer, and let $P$ be the path of length~$d$ given as 
\[
v,u_1,\ldots,u_{d-1},w.
\] 
Let $H^v$ and $H^w$ be the complete graphs isomorphic to $K_t$ and $K_s$, respectively, with $t,s\geq 0$. Also suppose that $\mathcal{H}_1,\ldots,\mathcal{H}_{d-1}$ as well as $\mathcal{H}_{1,2},\mathcal{H}_{2,3},\ldots,\mathcal{H}_{d-2,d-1}$ are some families of complete graphs on disjoint sets of vertices such that $|V(\mathcal{H}_i)|\geq 0$ and $|V(\mathcal{H}_{j,j+1})|\geq 0$ for each $i=1,\ldots,d-1$ and $j=1,\ldots,d-2$. Then we construct a graph $H$ on the vertex set 
\[
\{v,u_1,\ldots,u_{d-1},w\}\cup V(H^v)\cup V(H^w)\cup (\bigcup_{i=1}^{d-1}V(\mathcal{H}_i))\cup (\bigcup_{j=1}^{d-2}V(\mathcal{H}_{j,j+1})),
\]     
such that 
\[
N_H(v)=\{u_1\}\cup V({H}^v),
\]

\[
N_H(w)=\{u_{d-1}\}\cup V({H}^w),
\]

\[
N_H(u_1)=\{v,u_1\}\cup V(H^v)\cup V(\mathcal{H}_1)\cup V(\mathcal{H}_{1,2}),
\]

\[
N_H(u_{d-1})=\{w,u_{d-2}\}\cup V(H^w)\cup V(\mathcal{H}_{d-1})\cup V(\mathcal{H}_{d-2,d-1}),
\]
and 
\[
N_H(u_i)=\{u_{i-1},u_{i+1}\}\cup V(\mathcal{H}_{i})\cup V(\mathcal{H}_{i-1,i})\cup V(\mathcal{H}_{i,i+1}),
\]
for all $i=2,\ldots, d-2$. The induced subgraph of $H$ on the vertex set 
\[
(\bigcup_{i=1}^{d-1}V(\mathcal{H}_i))\cup (\bigcup_{j=1}^{d-2}V(\mathcal{H}_{j,j+1}))
\] 
is just a disjoint union of some complete graphs which are exactly the elements of the families $\mathcal{H}_i$'s and $\mathcal{H}_{j,j+1}$'s. In particular, the path $P$ is an induced path of $H$. 

Now, let $\mathcal{G}_d$ be the family of such graphs~$H$ constructed as above. Then, it is easily seen that for any $H\in \mathcal{G}_d$, one has: 
\[
V(H)\geq 3, \, \, \, \mathrm{diam}(H)=d\geq 2, \, \, \, \kappa(H)=1, \, \, \, f(H)\geq 2, \, \, \  f(H)=|V(H)|-(d-1).
\]

\begin{Example}\label{Example-G7}
	{\em 
	The graph depicted in Figure~\ref{G7} belongs to the family $\mathcal{G}_7$ of graphs. Using the notation of our construction, in this case, we have: 
	\begin{itemize}
		\item $H^v$ is isomorphic to the complete graph $K_2$;
		\item $H^w$ is the empty graph (i.e. the graph with no vertices and no edges);
		\item $\mathcal{H}_1$ consists of two elements, each of them isomorphic to $K_1$;
		\item $\mathcal{H}_5$ consists of four elements, one isomorphic to $K_2$, and three isomorphic to~$K_1$;
		\item $\mathcal{H}_{2,3}$ consists of two elements, one isomorphic to $K_2$, and one isomorphic to~$K_1$;
		\item $\mathcal{H}_{3,4}$ consists of one elements isomorphic to $K_2$;
		\item $\mathcal{H}_{i}$ and $\mathcal{H}_{j,j+1}$ are empty for all other $i$ and $j$ which are not mentioned above. 
	\end{itemize}
}
\end{Example}  

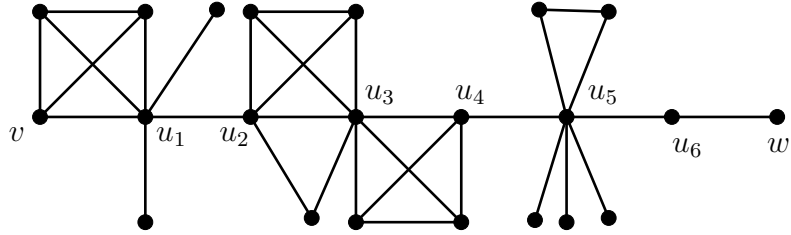
\begin{figure}[h!]
	\centering
	\begin{tikzpicture}[scale=1.4]
\draw [line width=1pt] (-7,2)-- (-7,1);
\draw [line width=1pt] (-7,2)-- (-6,2);
\draw [line width=1pt] (-6,2)-- (-6,1);
\draw [line width=1pt] (-6,1)-- (-7,1);
\draw [line width=1pt] (-6,2)-- (-7,1);
\draw [line width=1pt] (-7,2)-- (-6,1);
\draw [line width=1pt] (-6,1)-- (-5.32,2.02);
\draw [line width=1pt] (-6,1)-- (-6,0);
\draw [line width=1pt] (-6,1)-- (-5,1);
\draw [line width=1pt] (-5,1)-- (-4,1);
\draw [line width=1pt] (-4,1)-- (-3,1);
\draw [line width=1pt] (-3,1)-- (-2,1);
\draw [line width=1pt] (-2,1)-- (-1,1);
\draw [line width=1pt] (-1,1)-- (0,1);
\draw [line width=1pt] (-5,2)-- (-4,2);
\draw [line width=1pt] (-4,1)-- (-4,2);
\draw [line width=1pt] (-5,2)-- (-5,1);
\draw [line width=1pt] (-5,1)-- (-4,2);
\draw [line width=1pt] (-5,2)-- (-4,1);
\draw [line width=1pt] (-4,1)-- (-4.42,0.04);
\draw [line width=1pt] (-4.42,0.04)-- (-5,1);
\draw [line width=1pt] (-4,1)-- (-4,0);
\draw [line width=1pt] (-4,0)-- (-3,0);
\draw [line width=1pt] (-3,0)-- (-3,1);
\draw [line width=1pt] (-4,0)-- (-3,1);
\draw [line width=1pt] (-4,1)-- (-3,0);
\draw [line width=1pt] (-2.26,2.02)-- (-1.6,2);
\draw [line width=1pt] (-2,1)-- (-2.26,2.02);
\draw [line width=1pt] (-1.6,2)-- (-2,1);
\draw [line width=1pt] (-2,1)-- (-2,0);
\draw [line width=1pt] (-2,1)-- (-2.3,0.02);
\draw [line width=1pt] (-2,1)-- (-1.6,0.04);
\draw (-7.4,1) node[anchor=north west] {$v$};
\draw (-6,1) node[anchor=north west] {$u_1$};
\draw (-5.4,1) node[anchor=north west] {$u_2$};
\draw (-4.02,1.4) node[anchor=north west] {$u_3$};
\draw (-3.16,1.4) node[anchor=north west] {$u_4$};
\draw (-1.9,1.4) node[anchor=north west] {$u_5$};
\draw (-1.1,0.9) node[anchor=north west] {$u_6$};
\draw (-0.2,0.9) node[anchor=north west] {$w$};
\begin{scriptsize}
\draw [fill=black] (-7,2) circle (2pt);
\draw [fill=black] (-7,1) circle (2pt);
\draw [fill=black] (-6,2) circle (2pt);
\draw [fill=black] (-6,1) circle (2pt);
\draw [fill=black] (-5,1) circle (2pt);
\draw [fill=black] (-4,1) circle (2pt);
\draw [fill=black] (-3,1) circle (2pt);
\draw [fill=black] (-2,1) circle (2pt);
\draw [fill=black] (-1,1) circle (2pt);
\draw [fill=black] (0,1) circle (2pt);
\draw [fill=black] (-5.32,2.02) circle (2pt);
\draw [fill=black] (-6,0) circle (2pt);
\draw [fill=black] (-4.42,0.04) circle (2pt);
\draw [fill=black] (-5,2) circle (2pt);
\draw [fill=black] (-4,2) circle (2pt);
\draw [fill=black] (-4,0) circle (2pt);
\draw [fill=black] (-3,0) circle (2pt);
\draw [fill=black] (-2.3,0.02) circle (2pt);
\draw [fill=black] (-2,0) circle (2pt);
\draw [fill=black] (-1.6,0.04) circle (2pt);
\draw [fill=black] (-2.26,2.02) circle (2pt);
\draw [fill=black] (-1.6,2) circle (2pt);
\end{scriptsize}
\end{tikzpicture}
	\caption{A graph in the family $\mathcal{G}_7$}
	\label{G7}
\end{figure}

All the graphs, with up to five vertices, in the families $\mathcal{G}_2$ and $\mathcal{G}_3$ are shown in Figure~\ref{G2} and Figure~\ref{G3}, respectively. 

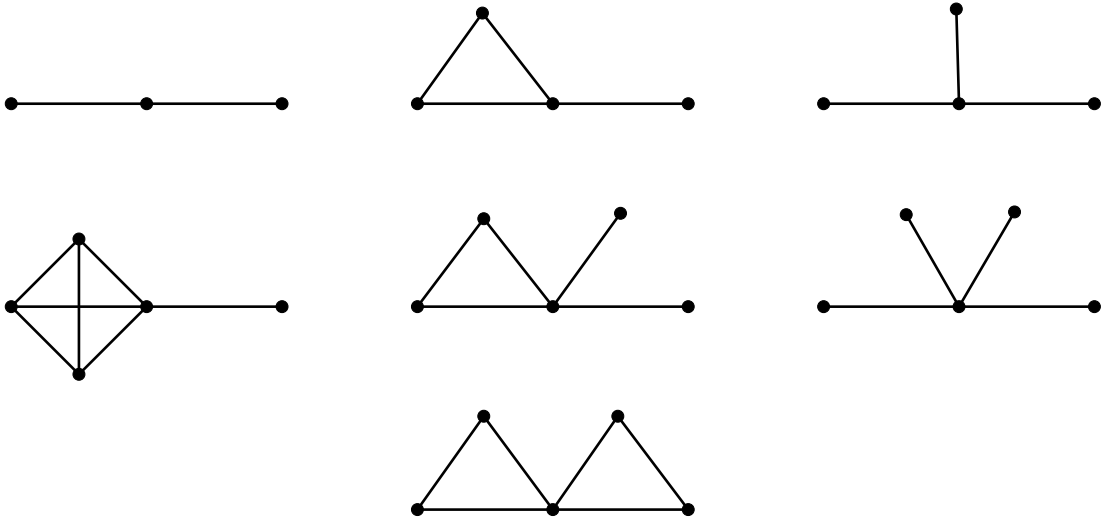
\begin{figure}[h!]
	\centering
	\begin{tikzpicture}[scale=0.9]
\draw [line width=1pt] (0,1)-- (2,1);
\draw [line width=1pt] (2,1)-- (4,1);
\draw [line width=1pt] (6.96,2.34)-- (6,1);
\draw [line width=1pt] (6,1)-- (8,1);
\draw [line width=1pt] (8,1)-- (6.96,2.34);
\draw [line width=1pt] (8,1)-- (10,1);
\draw [line width=1pt] (12,1)-- (14,1);
\draw [line width=1pt] (14,1)-- (16,1);
\draw [line width=1pt] (14,1)-- (13.96,2.4);
\draw [line width=1pt] (1,-1)-- (0,-2);
\draw [line width=1pt] (0,-2)-- (1,-3);
\draw [line width=1pt] (1,-3)-- (2,-2);
\draw [line width=1pt] (2,-2)-- (1,-1);
\draw [line width=1pt] (1,-1)-- (1,-3);
\draw [line width=1pt] (0,-2)-- (2,-2);
\draw [line width=1pt] (2,-2)-- (4,-2);
\draw [line width=1pt] (6.98,-0.7)-- (6,-2);
\draw [line width=1pt] (8,-2)-- (6,-2);
\draw [line width=1pt] (6.98,-0.7)-- (8,-2);
\draw [line width=1pt] (8,-2)-- (9,-0.62);
\draw [line width=1pt] (8,-2)-- (10,-2);
\draw [line width=1pt] (12,-2)-- (14,-2);
\draw [line width=1pt] (14,-2)-- (16,-2);
\draw [line width=1pt] (13.22,-0.64)-- (14,-2);
\draw [line width=1pt] (14.82,-0.6)-- (14,-2);
\draw [line width=1pt] (6,-5)-- (6.98,-3.62);
\draw [line width=1pt] (6.98,-3.62)-- (8,-5);
\draw [line width=1pt] (8,-5)-- (6,-5);
\draw [line width=1pt] (8,-5)-- (8.96,-3.62);
\draw [line width=1pt] (8.96,-3.62)-- (10,-5);
\draw [line width=1pt] (10,-5)-- (8,-5);
\begin{scriptsize}
\draw [fill=black] (0,1) circle (2.5pt);
\draw [fill=black] (2,1) circle (2.5pt);
\draw [fill=black] (4,1) circle (2.5pt);
\draw [fill=black] (6,1) circle (2.5pt);
\draw [fill=black] (8,1) circle (2.5pt);
\draw [fill=black] (10,1) circle (2.5pt);
\draw [fill=black] (6.96,2.34) circle (2.5pt);
\draw [fill=black] (12,1) circle (2.5pt);
\draw [fill=black] (14,1) circle (2.5pt);
\draw [fill=black] (16,1) circle (2.5pt);
\draw [fill=black] (13.96,2.4) circle (2.5pt);
\draw [fill=black] (0,-2) circle (2.5pt);
\draw [fill=black] (1,-1) circle (2.5pt);
\draw [fill=black] (2,-2) circle (2.5pt);
\draw [fill=black] (1,-3) circle (2.5pt);
\draw [fill=black] (4,-2) circle (2.5pt);
\draw [fill=black] (6,-2) circle (2.5pt);
\draw [fill=black] (8,-2) circle (2.5pt);
\draw [fill=black] (10,-2) circle (2.5pt);
\draw [fill=black] (6.98,-0.7) circle (2.5pt);
\draw [fill=black] (9,-0.62) circle (2.5pt);
\draw [fill=black] (12,-2) circle (2.5pt);
\draw [fill=black] (14,-2) circle (2.5pt);
\draw [fill=black] (16,-2) circle (2.5pt);
\draw [fill=black] (13.22,-0.64) circle (2.5pt);
\draw [fill=black] (14.82,-0.6) circle (2.5pt);
\draw [fill=black] (6,-5) circle (2.5pt);
\draw [fill=black] (6.98,-3.62) circle (2.5pt);
\draw [fill=black] (8,-5) circle (2.5pt);
\draw [fill=black] (8.96,-3.62) circle (2.5pt);
\draw [fill=black] (10,-5) circle (2.5pt);
\end{scriptsize}
\end{tikzpicture}
\caption{The graphs with up to five vertices in $\mathcal{G}_2$}
\label{G2}
\end{figure}

\begin{figure}[h!]
	\centering
	\begin{tikzpicture}[scale=0.9]
\draw [line width=1pt] (-3,2)-- (-1,2);
\draw [line width=1pt] (-1,2)-- (1,2);
\draw [line width=1pt] (1,2)-- (3,2);
\draw [line width=1pt] (5,2)-- (7,2);
\draw [line width=1pt] (7,2)-- (9,2);
\draw [line width=1pt] (9,2)-- (11,2);
\draw [line width=1pt] (7,2)-- (6.98,3.36);
\draw [line width=1pt] (-3,-1)-- (-1,-1);
\draw [line width=1pt] (-1,-1)-- (1,-1);
\draw [line width=1pt] (1,-1)-- (3,-1);
\draw [line width=1pt] (-3,-1)-- (-2,0.36);
\draw [line width=1pt] (-2,0.36)-- (-1,-1);
\draw [line width=1pt] (5,-1)-- (7,-1);
\draw [line width=1pt] (7,-1)-- (9,-1);
\draw [line width=1pt] (9,-1)-- (11,-1);
\draw [line width=1pt] (7,-1)-- (8,0.32);
\draw [line width=1pt] (8,0.32)-- (9,-1);
\begin{scriptsize}
\draw [fill=black] (-3,2) circle (2.5pt);
\draw [fill=black] (-1,2) circle (2.5pt);
\draw [fill=black] (1,2) circle (2.5pt);
\draw [fill=black] (3,2) circle (2.5pt);
\draw [fill=black] (5,2) circle (2.5pt);
\draw [fill=black] (7,2) circle (2.5pt);
\draw [fill=black] (9,2) circle (2.5pt);
\draw [fill=black] (11,2) circle (2.5pt);
\draw [fill=black] (6.98,3.36) circle (2.5pt);
\draw [fill=black] (-3,-1) circle (2.5pt);
\draw [fill=black] (-1,-1) circle (2.5pt);
\draw [fill=black] (1,-1) circle (2.5pt);
\draw [fill=black] (3,-1) circle (2.5pt);
\draw [fill=black] (-2,0.36) circle (2.5pt);
\draw [fill=black] (5,-1) circle (2.5pt);
\draw [fill=black] (7,-1) circle (2.5pt);
\draw [fill=black] (9,-1) circle (2.5pt);
\draw [fill=black] (11,-1) circle (2.5pt);
\draw [fill=black] (8,0.32) circle (2.5pt);
\end{scriptsize}
\end{tikzpicture}
\caption{The graphs with up to five vertices in $\mathcal{G}_3$}
\label{G3}
\end{figure}
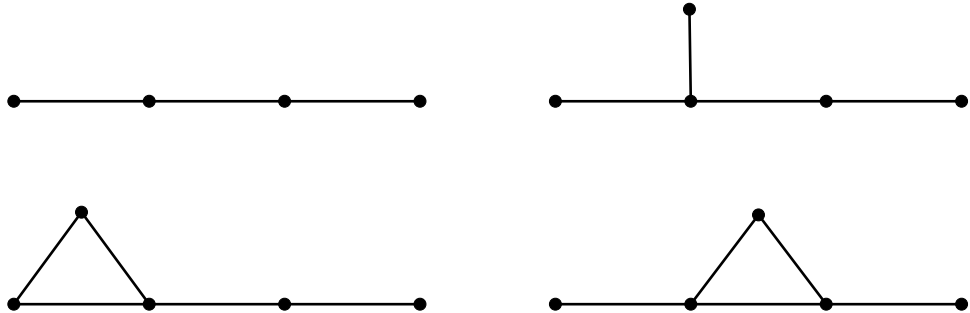

Before constructing the next families, we need to recall the concept of the \emph{join product} of graphs. Let $H_1$ and $H_2$ be two graphs on disjoint sets of vertices $V(H_1)$ and $V(H_2)$. Then the join product of $H_1$ and $H_2$, denoted by $H_1*H_2$, is
the graph with the vertex set $V(H_1)\cup V(H_2)$, and the edge set
\[
E(H_1)\cup E(H_2)\cup \{\{v,w\}~:~v\in V(H_1),~w\in V(H_2)\}.
\]

\medskip
\noindent
{\bf The families $\mathcal{F}_q$:} Suppose that $q\geq 2$ is an integer. Let $\mathcal{F}_q$ be the family of graphs consisting of all graphs of the form 
\[
K_q*(F_1\cup \cdots \cup F_p)
\]   
with $p\geq 2$, where $K_q$ is the complete graph on $q$ vertices, and for each $i=1,\ldots,p$, the graph $F_i$ is the complete graph isomorphic to $K_{\ell_i}$ with $\ell_i\geq 1$ and 
\[
V(F_i)\cap V(F_j)=\emptyset
\] 
for each $i,j$ with $1\leq i<j\leq p$. Then, by our construction, it is easily seen that for any graph $G\in \mathcal{F}_q$, one has: 
\[
|V(G)|\geq 4, \, \, \, \mathrm{diam}(G)=2, \, \, \, \kappa(G)=q\geq 2, \, \, \, f(G)\geq 2, \, \, \  f(G)=|V(G)|-q.
\]  

\begin{Example}\label{Example-F3}
	{\em 
		The graph shown in Figure~\ref{F3} belongs to the family $\mathcal{F}_3$ of graphs. In this case, we have the join product of $K_3$ and the three graphs~$F_1$, $F_2$ (which are isomorphic to $K_1$) and $F_3$ (which is isomorphic to $K_2$), as shown in Figure~\ref{F3}. 
	}
\end{Example}

\definecolor{qqqqff}{rgb}{0,0,1}
\definecolor{fuqqzz}{rgb}{0.9568627450980393,0,0.6}
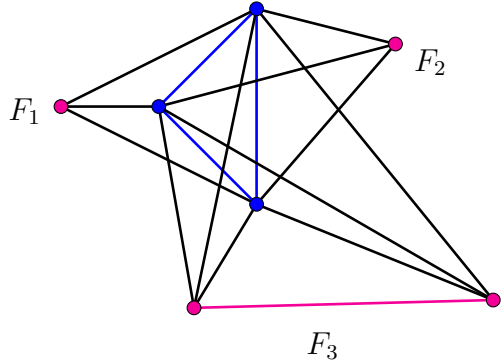
\begin{figure}[h!]
	\centering
	\begin{tikzpicture}[scale=1.3]
\draw [line width=1pt,color=qqqqff] (2,2)-- (1,1);
\draw [line width=1pt,color=qqqqff] (1,1)-- (2,0);
\draw [line width=1pt,color=qqqqff] (2,2)-- (2,0);
\draw [line width=1pt] (0,1)-- (1,1);
\draw [line width=1pt] (0,1)-- (2,2);
\draw [line width=1pt] (0,1)-- (2,0);
\draw [line width=1pt] (3.42,1.64)-- (2,2);
\draw [line width=1pt] (3.42,1.64)-- (2,0);
\draw [line width=1pt] (3.42,1.64)-- (1,1);
\draw [line width=1pt,color=fuqqzz] (1.36,-1.06)-- (4.42,-0.98);
\draw [line width=1pt] (4.42,-0.98)-- (2,0);
\draw [line width=1pt] (1.36,-1.06)-- (2,0);
\draw [line width=1pt] (1.36,-1.06)-- (1,1);
\draw [line width=1pt] (1.36,-1.06)-- (2,2);
\draw [line width=1pt] (4.42,-0.98)-- (1,1);
\draw [line width=1pt] (4.42,-0.98)-- (2,2);
\draw (-0.65,1.2) node[anchor=north west] {$F_1$};
\draw (3.5,1.7) node[anchor=north west] {$F_2$};
\draw (2.4,-1.2) node[anchor=north west] {$F_3$};
\begin{scriptsize}
\draw [fill=fuqqzz] (0,1) circle (2pt);
\draw [fill=qqqqff] (1,1) circle (2pt);
\draw [fill=qqqqff] (2,2) circle (2pt);
\draw [fill=qqqqff] (2,0) circle (2pt);
\draw [fill=fuqqzz] (3.42,1.64) circle (2pt);
\draw [fill=fuqqzz] (1.36,-1.06) circle (2pt);
\draw [fill=fuqqzz] (4.42,-0.98) circle (2pt);
\end{scriptsize}
\end{tikzpicture}
\caption{A graph in the family $\mathcal{F}_3$}
\label{F3}
\end{figure}

All the graphs with up to six vertices in the families $\mathcal{F}_2$ and $\mathcal{F}_3$ are shown in Figure~\ref{F2} and Figure~\ref{F3-1}, respectively. 

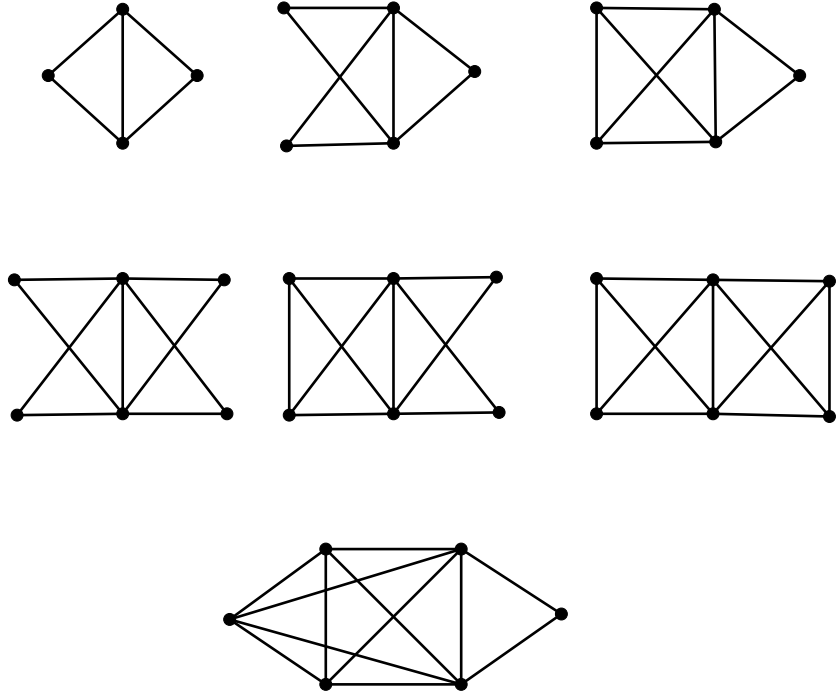
\begin{figure}[h!]
	\centering
	\begin{tikzpicture}[scale=0.9]
	\draw [line width=1pt] (-2.1,3)-- (-1,3.98);
	\draw [line width=1pt] (-2.1,3)-- (-1,2);
	\draw [line width=1pt] (-1,3.98)-- (-1,2);
	\draw [line width=1pt] (-1,2)-- (0.1,3);
	\draw [line width=1pt] (-1,3.98)-- (0.1,3);
	\draw [line width=1pt] (3,4)-- (3,2);
	\draw [line width=1pt] (1.42,1.96)-- (3,4);
	\draw [line width=1pt] (1.42,1.96)-- (3,2);
	\draw [line width=1pt] (3,4)-- (1.38,4);
	\draw [line width=1pt] (1.38,4)-- (3,2);
	\draw [line width=1pt] (4.2,3.06)-- (3,4);
	\draw [line width=1pt] (4.2,3.06)-- (3,2);
	\draw [line width=1pt] (6,2)-- (6,4);
	\draw [line width=1pt] (6,4)-- (7.74,3.98);
	\draw [line width=1pt] (7.76,2.02)-- (6,2);
	\draw [line width=1pt] (7.76,2.02)-- (7.74,3.98);
	\draw [line width=1pt] (6,2)-- (7.74,3.98);
	\draw [line width=1pt] (6,4)-- (7.76,2.02);
	\draw [line width=1pt] (9,3)-- (7.74,3.98);
	\draw [line width=1pt] (9,3)-- (7.76,2.02);
	\draw [line width=1pt] (-1,0)-- (-1,-2);
	\draw [line width=1pt] (-2.6,-0.02)-- (-1,0);
	\draw [line width=1pt] (-2.6,-0.02)-- (-1,-2);
	\draw [line width=1pt] (-1,-2)-- (-2.56,-2.02);
	\draw [line width=1pt] (-2.56,-2.02)-- (-1,0);
	\draw [line width=1pt] (-1,0)-- (0.5,-0.02);
	\draw [line width=1pt] (0.5,-0.02)-- (-1,-2);
	\draw [line width=1pt] (-1,-2)-- (0.54,-2);
	\draw [line width=1pt] (0.54,-2)-- (-1,0);
	\draw [line width=1pt] (3,0)-- (3,-2);
	\draw [line width=1pt] (1.46,-2.02)-- (1.46,0);
	\draw [line width=1pt] (1.46,0)-- (3,0);
	\draw [line width=1pt] (3,-2)-- (1.46,-2.02);
	\draw [line width=1pt] (1.46,-2.02)-- (3,0);
	\draw [line width=1pt] (1.46,0)-- (3,-2);
	\draw [line width=1pt] (3,-2)-- (4.56,-1.98);
	\draw [line width=1pt] (4.56,-1.98)-- (3,0);
	\draw [line width=1pt] (3,0)-- (4.52,0.02);
	\draw [line width=1pt] (4.52,0.02)-- (3,-2);
	\draw [line width=1pt] (6,0)-- (6,-2);
	\draw [line width=1pt] (6,-2)-- (7.72,-2);
	\draw [line width=1pt] (7.72,-2)-- (7.72,-0.02);
	\draw [line width=1pt] (7.72,-0.02)-- (6,0);
	\draw [line width=1pt] (6,0)-- (7.72,-2);
	\draw [line width=1pt] (6,-2)-- (7.72,-0.02);
	\draw [line width=1pt] (7.72,-0.02)-- (9.44,-0.04);
	\draw [line width=1pt] (9.44,-0.04)-- (9.44,-2.04);
	\draw [line width=1pt] (9.44,-2.04)-- (7.72,-2);
	\draw [line width=1pt] (7.72,-2)-- (9.44,-0.04);
	\draw [line width=1pt] (7.72,-0.02)-- (9.44,-2.04);
	\draw [line width=1pt] (4,-4)-- (4,-6);
	\draw [line width=1pt] (4,-4)-- (2,-4);
	\draw [line width=1pt] (2,-4)-- (2,-6);
	\draw [line width=1pt] (2,-6)-- (4,-6);
	\draw [line width=1pt] (2,-4)-- (0.58,-5.04);
	\draw [line width=1pt] (0.58,-5.04)-- (2,-6);
	\draw [line width=1pt] (4,-4)-- (5.48,-4.96);
	\draw [line width=1pt] (5.48,-4.96)-- (4,-6);
	\draw [line width=1pt] (0.58,-5.04)-- (4,-4);
	\draw [line width=1pt] (0.58,-5.04)-- (4,-6);
	\draw [line width=1pt] (2,-4)-- (4,-6);
	\draw [line width=1pt] (2,-6)-- (4,-4);
	\begin{scriptsize}
	\draw [fill=black] (-2.1,3) circle (2.5pt);
	\draw [fill=black] (-1,3.98) circle (2.5pt);
	\draw [fill=black] (-1,2) circle (2.5pt);
	\draw [fill=black] (0.1,3) circle (2.5pt);
	\draw [fill=black] (3,4) circle (2.5pt);
	\draw [fill=black] (3,2) circle (2.5pt);
	\draw [fill=black] (1.42,1.96) circle (2.5pt);
	\draw [fill=black] (1.38,4) circle (2.5pt);
	\draw [fill=black] (4.2,3.06) circle (2.5pt);
	\draw [fill=black] (6,2) circle (2.5pt);
	\draw [fill=black] (6,4) circle (2.5pt);
	\draw [fill=black] (7.74,3.98) circle (2.5pt);
	\draw [fill=black] (7.76,2.02) circle (2.5pt);
	\draw [fill=black] (9,3) circle (2.5pt);
	\draw [fill=black] (-1,0) circle (2.5pt);
	\draw [fill=black] (-1,-2) circle (2.5pt);
	\draw [fill=black] (-2.6,-0.02) circle (2.5pt);
	\draw [fill=black] (-2.56,-2.02) circle (2.5pt);
	\draw [fill=black] (0.5,-0.02) circle (2.5pt);
	\draw [fill=black] (0.54,-2) circle (2.5pt);
	\draw [fill=black] (3,0) circle (2.5pt);
	\draw [fill=black] (3,-2) circle (2.5pt);
	\draw [fill=black] (1.46,-2.02) circle (2.5pt);
	\draw [fill=black] (1.46,0) circle (2.5pt);
	\draw [fill=black] (4.56,-1.98) circle (2.5pt);
	\draw [fill=black] (4.52,0.02) circle (2.5pt);
	\draw [fill=black] (6,0) circle (2.5pt);
	\draw [fill=black] (6,-2) circle (2.5pt);
	\draw [fill=black] (7.72,-2) circle (2.5pt);
	\draw [fill=black] (7.72,-0.02) circle (2.5pt);
	\draw [fill=black] (9.44,-0.04) circle (2.5pt);
	\draw [fill=black] (9.44,-2.04) circle (2.5pt);
	\draw [fill=black] (4,-4) circle (2.5pt);
	\draw [fill=black] (4,-6) circle (2.5pt);
	\draw [fill=black] (2,-4) circle (2.5pt);
	\draw [fill=black] (2,-6) circle (2.5pt);
	\draw [fill=black] (0.58,-5.04) circle (2.5pt);
	\draw [fill=black] (5.48,-4.96) circle (2.5pt);
	\end{scriptsize}
	\end{tikzpicture}
	\caption{The graphs with up to six vertices in $\mathcal{F}_2$}
	\label{F2}
\end{figure}

\begin{figure}[h!]
	\centering
	\begin{tikzpicture}[scale=1.1]
	\draw [line width=1pt] (-1,3)-- (-2,2);
	\draw [line width=1pt] (0,2)-- (-2,2);
	\draw [line width=1pt] (-1,3)-- (0,2);
	\draw [line width=1pt] (-2.24,3.22)-- (-2,2);
	\draw [line width=1pt] (-2.24,3.22)-- (-1,3);
	\draw [line width=1pt] (-2.24,3.22)-- (0,2);
	\draw [line width=1pt] (0.2,3.26)-- (-1,3);
	\draw [line width=1pt] (0.2,3.26)-- (0,2);
	\draw [line width=1pt] (0.2,3.26)-- (-2,2);
	\draw [line width=1pt] (4,3)-- (3,2);
	\draw [line width=1pt] (5,2)-- (3,2);
	\draw [line width=1pt] (5,2)-- (4,3);
	\draw [line width=1pt] (5.22,3.22)-- (4,3);
	\draw [line width=1pt] (5.22,3.22)-- (5,2);
	\draw [line width=1pt] (5.22,3.22)-- (3,2);
	\draw [line width=1pt] (2.8,3.2)-- (3,2);
	\draw [line width=1pt] (2.8,3.2)-- (5,2);
	\draw [line width=1pt] (2.8,3.2)-- (4,3);
	\draw [line width=1pt] (4,1)-- (5,2);
	\draw [line width=1pt] (4,1)-- (3,2);
	\draw [line width=1pt] (4,1)-- (4,3);
	\draw [line width=1pt] (8,2)-- (8,3);
	\draw [line width=1pt] (8.98,3.5)-- (8,3);
	\draw [line width=1pt] (8.98,3.5)-- (10,3);
	\draw [line width=1pt] (10,3)-- (10,2);
	\draw [line width=1pt] (10,2)-- (8,2);
	\draw [line width=1pt] (9,1)-- (10,3);
	\draw [line width=1pt] (8,2)-- (10,3);
	\draw [line width=1pt] (9,1)-- (10,2);
	\draw [line width=1pt] (9,1)-- (8,2);
	\draw [line width=1pt] (8,3)-- (10,2);
	\draw [line width=1pt] (8,3)-- (10,3);
	\draw [line width=1pt] (8.98,3.5)-- (8,2);
	\draw [line width=1pt] (8.98,3.5)-- (10,2);
	\begin{scriptsize}
	\draw [fill=black] (-1,3) circle (2.5pt);
	\draw [fill=black] (-2,2) circle (2.5pt);
	\draw [fill=black] (0,2) circle (2.5pt);
	\draw [fill=black] (-2.24,3.22) circle (2.5pt);
	\draw [fill=black] (0.2,3.26) circle (2.5pt);
	\draw [fill=black] (4,3) circle (2.5pt);
	\draw [fill=black] (3,2) circle (2.5pt);
	\draw [fill=black] (5,2) circle (2.5pt);
	\draw [fill=black] (5.22,3.22) circle (2.5pt);
	\draw [fill=black] (2.8,3.2) circle (2.5pt);
	\draw [fill=black] (4,1) circle (2.5pt);
	\draw [fill=black] (8,2) circle (2.5pt);
	\draw [fill=black] (8,3) circle (2.5pt);
	\draw [fill=black] (8.98,3.5) circle (2.5pt);
	\draw [fill=black] (10,3) circle (2.5pt);
	\draw [fill=black] (10,2) circle (2.5pt);
	\draw [fill=black] (9,1) circle (2.5pt);
	\end{scriptsize}
\end{tikzpicture}
\caption{The graphs with up to six vertices in $\mathcal{F}_3$}
\label{F3-1}
\end{figure}
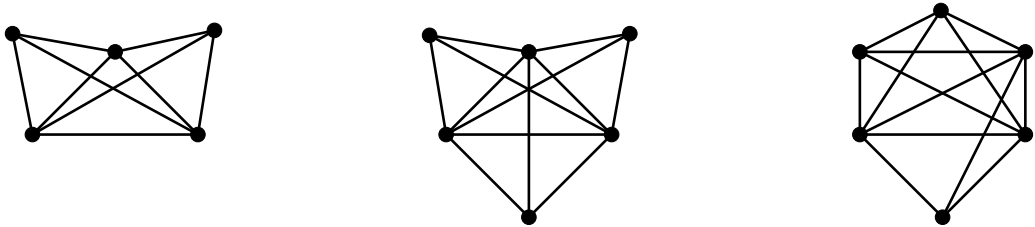
	
\medskip
Theorem~\ref{Boston} together with the {\bf{First Step}} and the {\bfseries{Second Step}} imply that any finite simple connected graph $G$ with $n\geq 3$ vertices for which $f(G)+\mathrm{diam}(G)=n+2-\kappa(G)$, is a graph in the family $\mathcal{G}_{\mathrm{diam}(G)}$ or the family $\mathcal{F}_{\kappa(G)}$. Therefore, we get the following desired classification: 
   
\begin{Theorem}
	\label{classification}
	Let $G$ be a finite simple non-complete connected graph on 
	$[n]$. Then the following statements are equivalent:
	\begin{enumerate}
		\item[{\em(a)}] $f(G) + \diam(G) = n + 2 - \kappa(G)$.
		\item[{\em(b)}]$G\in \mathcal{G}_{\mathrm{diam}(G)}$ or $G\in \mathcal{F}_{\kappa(G)}$.
	\end{enumerate}
\end{Theorem}

\medskip
Returning to an initial motivation arising from the depth of binomial edge ideals, we have the following corollary:

\begin{Corollary}\label{depth}
 Let $d\geq 2$ and $q\geq 2$ be integers. If $G\in \mathcal{G}_d$ and 
 $H\in \mathcal{F}_q$ are graphs with $n$ vertices, then 
 \[
 \depth (S/J_G)=n+1,  
 \]
 \[
 \depth (S/J_H)=n+2-q.
 \]
\end{Corollary}  

\begin{Remark}\label{Kumar}
	{\em 
		In the case of graphs in the families $\mathcal{F}_q$ for any $q\geq 2$, the formula given in Corollary~\ref{depth} for the depth of the binomial edge ideal also follows from \cite[Theorem~3.4 and Theorem~3.9]{KuSa}. 
	}
\end{Remark}


\begin{Remark}\label{block-closed}
	{\em 
		We would like to remark that some precise formulas for the depth of binomial edge ideals of some classes of chordal graphs are known in the literature. Among them are \emph{(generalized) block} graphs, (see~\cite{EHH} and \cite{KiSa}) and some subclasses of \emph{closed} graphs (also known as \emph{proper interval} graphs), (see~\cite{AlHo}). 
	
Although the families $\mathcal{G}_d$ of chordal graphs, for any $d\geq 2$, have intersections with the above classes, there are infinitely many chordal graphs in each $\mathcal{G}_d$ for $d\geq 3$ which do not belong to any of the above classes. Figure~\ref{G7} depicts an example of such graphs.     
}
\end{Remark}

\end{document}